\documentclass[12pt,leqno]{report}
\usepackage{amssymb,amsmath,amsfonts,amsbsy, amsthm, xspace, latexsym, amscd, graphicx, fancyhdr,verbatim,enumitem}
\usepackage[colorlinks=true, pdfstartview=FitV, linkcolor=blue, citecolor=blue, urlcolor=blue]{hyperref}

\usepackage[normalem]{ulem}
\usepackage{color}

\swapnumbers
\theoremstyle{definition}
\newtheorem{definition}{Definition}[section]
\newtheorem{remark}[definition]{Remark}
\newtheorem{example}[definition]{Example}
\newtheorem{examples}[definition]{Examples}

\theoremstyle{plain}
\newtheorem{theorem}[definition]{Theorem}

\newtheorem{lemma}[definition]{Lemma}
\newtheorem{corollary}[definition]{Corollary}

\newcommand{\bra}{\ensuremath{\big\langle}}
\newcommand{\ket}{\ensuremath{\big\rangle}}

\newcommand{\card}{\ensuremath{{\rm{card}}}}

\newlist{thm-enum}{enumerate}{2}
\newlist{pf-enum}{enumerate}{2}
\newlist{def-enum}{enumerate}{2}
\newlist{ex-enum}{enumerate}{2}
\newlist{rmk-enum}{enumerate}{2}
\setlist[thm-enum,1]{label=\textbf{\em (\roman*)},leftmargin=*,labelindent=.1\parindent}
\setlist[thm-enum,2]{label=\textbf{\em (\alph*)},leftmargin=*,labelindent=.1\parindent}
\setlist[pf-enum,1]{label=(\roman*),leftmargin=*,labelindent=.1\parindent}
\setlist[pf-enum,2]{label=(\alph*),leftmargin=*,labelindent=.1\parindent}
\setlist[def-enum,1]{label=\textbf{(\arabic*)},leftmargin=*,labelindent=.15\parindent}
\setlist[def-enum,2]{label=\textbf{(\alph*)},leftmargin=*,labelindent=.1\parindent}
\setlist[ex-enum,1]{label=\textbf{\arabic*.},leftmargin=*,labelindent=.15\parindent}
\setlist[rmk-enum,1]{label=\textbf{\arabic*.},leftmargin=*,labelindent=.15\parindent}

\newcounter{axiom-counter}

\title{Filter Extension and Ultrafilter Characterization}
\author{Max B. Garcia}

\begin{document}

\title{Filters and Ultrafilters in Real Analysis}
\author{Max Garcia\\ 
                        Mathematics Department\\                
                        California Polytechnic State University\\
                        San Luis Obispo, California 93407, USA\\
E-mail: mgarci78@calpoly.edu}
\date{}
\maketitle
\begin{abstract} We study free filters and their maximal extensions on the set of natural numbers. We characterize the limit of a sequence of real numbers in terms of the Fr\'{e}chet filter, which involves only one quantifier as opposed to the three non-commuting quantifiers in the usual definition. We construct the field of  real non-standard numbers and study their properties. We characterize the limit of a sequence of real numbers in terms of non-standard numbers which only requires a single quantifier as well. We are trying to make the point that the involvement of filters and/or non-standard numbers leads to a reduction in the number of quantifiers and hence, simplification, compared to the more traditional $\varepsilon, \delta$-definition of limits in real analysis.

\begin{center}
\line(1,0){250}
\end{center}

\noindent \textbf{Keywords and phrases:} Limit, sequential approach to real analysis, filter, Fr\'{e}chet filter, ultrafilter, non-standard analysis, reduction of quantifiers, infinitesimals, monad, internal sets.

\noindent \textbf{AMS Subject Classification:} 03C10, 03C20, 03C50, 03H05, 12L10, 26E35, 26A03, 26A06, 30G06\noindent \textbf{Keywords and phrases:} Limit, sequential approach to real analysis, filter, Fr\'{e}chet filter, ultrafilter, non-standard analysis, reduction of quantifiers, infinitesimals, monad, internal sets.

\noindent \textbf{AMS Subject Classification:} 03C10, 03C20, 03C50, 03H05, 12L10, 26E35, 26A03, 26A06, 30G06
\end{abstract}

\tableofcontents
\setcounter{tocdepth}{3}

\newpage

\pagenumbering{arabic}


\section*{Introduction}
\addcontentsline{toc}{section}{Introduction}

 In the \emph{sequential approach to real analysis} the definition:
\begin{equation}\label{E: General Limit}
 (\forall\varepsilon\in\mathbb R_+)(\exists\delta\in\mathbb R_+)(\forall x\in X)(0 <|x-r|<\delta \Rightarrow |f(x) -L|<\varepsilon),
\end{equation}
of the limit $\lim_{x\rightarrow r}f(x) = L$ can be  reduced to the definition:
\begin{equation}\label{E: Limit}
 (\forall\varepsilon\in\mathbb R_+)(\exists\nu\in\mathbb N)(\forall n\in\mathbb N)(n\geq\nu \Rightarrow |x_n-L|<\varepsilon).
\end{equation}
of the limit $\lim_{n\to\infty} x_n=L$ of a sequence in $\mathbb R$. Here is a \emph{summary of the sequential approach}:
\begin{enumerate}

\item  A sequence $(x_n)$ in a totally ordered field $\mathbb K$ is called \emph{convergent} if there exists $L\in\mathbb K$ such that $\lim_{n\to\infty} x_n=L$ in the sense of (\ref{E: Limit}).

\item A totally ordered field $\mathbb K$ is \emph{complete} if every fundamental (Cauchy) sequence in $\mathbb K$ is convergent. (For other characterization of completeness of an ordered field in terms of sequences, we refer to  (Hall~\cite{jHall}, Theorem 3.11.) All complete totally ordered fields are order isomorphic. We denote such a field by
 $\mathbb R$.
\item A point $r\in\mathbb R$ is a \emph{cluster point} of a set $X\subseteq\mathbb R$ if and only if there exists a sequence $(x_n)$ in $X$ such that: 
\begin{enumerate}
\item $x_n\not= r$ for all $n\in\mathbb N$.
\item $\lim_{n\to\infty} x_n=r$.
\end{enumerate}
We denote by $X_r^\mathbb N$ the set of all such sequences. 
\item Let $f: X\to\mathbb R$ be a real function. Then $\lim_{x\rightarrow r}f(x) = L$ if and only if  $\lim_{n\to\infty }f(x_n) = L$ for every sequence $(x_n)$ in $X_r^\mathbb N$.
\end{enumerate}
	We should mention that the equivalency between 3 and 4 above as well as the usual $\varepsilon,\delta$-approach, both require the involvement of the axiom of choice. We should note that although the sequential approach to real analysis is hardly new (Brannan~\cite{dBrannan}, Hewitt and K. Stromberg~\cite{hewitt} and Rudin~\cite{poma}), we are unaware of a systematical exposition written on this subject.

	The purpose of this project is to simplify the definition (\ref{E: Limit}) of $\lim_{n\to\infty} x_n=L$ (and thus to simplify the general definition (\ref{E: General Limit}) of $\lim_{x\rightarrow r}f(x) = L$) by reducing the number of quantifiers in (\ref{E: Limit}) from three to one. We achieve this by offering a characterization of the $\lim_{n\to\infty} x_n=L$ in terms of the Fr\'{e}chet filter. Using several examples, we demonstrate that our characterization of limit is convenient for proving the usual theorems in real analysis. We believe that our approach is simpler and more efficient than the conventional one. 

	In the second part of the project we extend the Fr\'{e}chet filter to a maximal filter (ultrafilter) and reproduce A. Robinson's~\cite{aRob66} characterization of limit in terms of non-standard numbers (for more accessible presentations of non-standard analysis we refer to: Cavalcante~\cite{cavalcante}, Davis~\cite{davis}, Keisler~\cite{jKeislerE}-\cite{jKeislerF}, Lindstr\o m~\cite{lindstrom}, Todorov~\cite{tdTod2000a}). We should emphasize that Robinson's characterization is again in terms of a single quantifier, and thus simpler and more elegant then the conventional $\varepsilon, \delta$-definition of limit. 

	In both characterizations -  in terms of Fr\'{e}chet's filter and in terms of non-standard numbers - we are trying to argue that it is quite possible to simplify the definition in real analysis by reducing the number of quantifiers in the definition, while still preserving the efficiency of the theory.

	Here is a more detailed description of the project.

	In Chapter 1, we present the basic definitions and properties of the free filters on $\mathbb N$ and their maximal extensions, commonly known as ultrafilters. 

	In Chapter 2, we  define what it means for a filter to be  Fr\'{e}chet as well as its characterizing properties. We then show how the  Fr\'{e}chet filter can be used to characterize limits in such a way that the number of quantifiers is reduced from 3 to 1.

	In Chapter 3, we build the non-standard numbers using the ultraproduct construction and show that ${^*}\mathbb R$ is a totally ordered field. We then characterize the numbers, sets and functions in  ${^*}\mathbb R$ and conclude by reproducing A. Robinson's characterization of limits in terms of non-standard numbers.


\chapter{Filters, Free Filters and Ultrafilters}

We begin with the basic theory of filters and ultrafilters defined on the natural numbers. These objects will serve as the foundation for our work in characterizing analysis under the  Fr\'echet filter as well as our construction of the nonstandard real numbers. After introducing the definitions for filters, free filters, and ultrafilters, we shall prove the existence of free ultrafilters and conclude with a discussion on the properties of ultrafilters. These last two sections will be of critical importance for our work in chapter 3.

\section{Filters and Ultrafilters}\label{S: Filters}
We present the basic definition for a filter and an ultrafilter on the set of natural numbers $\mathbb N$. In the following, $\mathcal{P}(\mathbb N)$ denotes the power set of $\mathbb N$.  For a more detailed exposition we refer to Davis~\cite{davis}.

\begin{definition}[Filters]\label{D: Filters} Let $\mathcal{F}$ be a non-empty subset of $\mathcal{P(\mathbb N)}$. 

	\begin{enumerate}
		\item We say $\mathcal {F}$ is a \textbf{filter} on $\mathbb N$ if:
 			\begin{enumerate}
				\item $\varnothing\notin\mathcal{F}$.
				\item $\mathcal{F}$ is closed under finite intersections, i.e.
					\[
						(\forall A,B\in\mathcal{P(\mathbb N)})(A, B\in\mathcal{F} \Rightarrow A\cap B\in\mathcal{F}).
					\] 
				\item Let $A\in\mathcal{F}$ and $B\in\mathcal{P(\mathbb N)}$. Then $A\subseteq B$ implies that $B\in\mathcal{F}$.
			\end{enumerate}
		\item A filter $\mathcal{F}$ is a \textbf{free filter} if:
			\begin{enumerate}
				\item[(d)] $\bigcap_{A\in\mathcal{F}}A=\varnothing$.
			\end{enumerate}
		\item A filter $\mathcal{F}$ is an \textbf{ultrafilter} or \textbf{maximal filter} if $\mathcal{F}$ is not properly contained in any other filter on $\mathbb N$.
	\end{enumerate}
\end{definition}

\begin{examples}[Filters]\label{Ex: Filters}
\hspace{1in}
	\begin{enumerate}
		 \item The Fr\'{e}chet filter $\mathcal{F}_r$ on $\mathbb N$ consists of the co-finite sets of $\mathbb N$, i.e.
			\[
				\mathcal{F}_r=\{S\in\mathcal{P}(\mathbb N): \mathbb N\setminus S \text{ is finite}\}.
			\]
			The Fr\'{e}chet filter is an example of a free filter that is not an ultrafilter.
		\item Let $B\subseteq\mathbb N.$ Then $\mathcal{F}=\{A\in\mathcal{P}(\mathbb N): B\subseteq A\}$ is a non free filter.
		\item Let $a\in\mathbb N.$ Then $\mathcal{F}_a=\{A\in\mathcal{P}(\mathbb N):a\in A\}$ is a non free ultrafilter.\newline
	\end{enumerate}

\end{examples}

If the reader would like to construct his/her own filter, they can accomplish this through the use of a filter basis on $\mathbb N$.

\begin{definition}[Filter Basis]
 	Let $G\subseteq\mathcal{P}(\mathbb N)$. Then $G$ is a filter basis on $\mathbb N$ if:
		\begin{enumerate}
			\item $\varnothing\not\in G$.
			\item $A,B\in G$ implies $A\cap B\in G$.
			\item $G\not=\varnothing$.
		\end{enumerate}
\end{definition}

\begin{theorem}
	If $G$ is a filter basis on $\mathbb N$, then there is a filter $\mathcal{F}$ on $\mathbb N$ such that $G\subseteq\mathcal{F}$.
\end{theorem}

\begin{proof}
	Let $\mathcal{F}=\{A\in\mathcal{P}(\mathbb N):C\subseteq A\text{ for some } C\in G\}$. Clearly $G\subseteq\mathcal{F}$. We conclude by showing that $\mathcal{F}$ is indeed a filter.
		\begin{description}
			\item[(i)] Suppose (to the contrary) that $\varnothing\in\mathcal{F}$. Then $C=\varnothing$, contradicting the fact that $G$ is a filter basis.
			\item[(ii)] Let $A,B\in\mathcal{F}$. Then $C\subseteq A\cap B$ and thus $A\cap B\in\mathcal{F}$.
			\item[(iii)] Let $A\in\mathcal{F}$ and let $B\in\mathcal{P}(\mathbb N)$ such that $A\subseteq B$. Then $C\subseteq B$. Thus $B\in\mathcal{F}$.
		\end{description}
\end{proof}


\section{Existence of Free Ultrafilters}\label{Sec: UltraExistence}
Though an explicit example of a free ultrafilter is not known, we can use the Axiom of Choice to prove that such a filter exists. The existence of free ultrafilters is of crucial importance to the construction of the non-standard real numbers.

\begin{theorem}
	Every free filter on $\mathbb N$ can be extended to a free ultrafilter on $\mathbb N$.
\end{theorem}

\begin{proof}
	Let $\mathcal{F}_0$ be a free filter on $\mathbb N$ and let $\mathcal{S}$ denote the set of free filters on $\mathbb N$ containing $\mathcal{F}_0$,
		\[
			\mathcal{S}=\{\mathcal{F}:\mathcal{F}_0 \subseteq \mathcal{F} \text{ and } \mathcal{F} \text{ is a free filter}\}.
		\]
	Observe that $\mathcal{S} \ne\varnothing$ since $\mathcal{F}_0\in\mathcal{S}$ a priori. We now partially order $\mathcal{S}$ by set inclusion. 
	Let $\mathcal{C}$ be a chain in $\mathcal{S}$, such that $\forall \mathcal{F}_i\in\mathcal{C}\text{ and } 
	\forall\mathcal{F}_j\in\mathcal{C}$, either $\mathcal{F}_i\subseteq\mathcal{F}_j\text{ or }\mathcal{F}_j\subseteq\mathcal{F}_i$.\newline

	Let $\Gamma=\bigcup_{\mathcal{F}\in\mathcal{C}}\mathcal{F}$. To show that $\Gamma\in\mathcal{S}$, we must prove that $\Gamma$ is a free filter. Indeed,
	\begin{description}
		\item[(a)] Suppose (to the contrary) that $\varnothing\in\Gamma$. Since
 		\begin{math}
 			\Gamma=\bigcup_{\mathcal{F}\in\mathcal{C}}\mathcal{F},
		\end{math}
		then $\varnothing\in\mathcal{F}$ for some $\mathcal{F}\in\mathcal{C}$, contradicting the fact that $\mathcal{F}$ is a free filter.

		\item[(b)] Let $\mathcal{X}$ and $\mathcal{Y}$ be elements in $\Gamma$. Then $\mathcal{X}\in\mathcal{F}_i$ and $\mathcal{Y}\in\mathcal{F}_j$ for some $\mathcal{F}_i$ and $\mathcal{F}_j$ in $\mathcal{C}$. Since $\mathcal{C}$ is a chain,  		
		$\mathcal{X}\cap\mathcal{Y}\in\mathcal{F}_i\cup\mathcal{F}_j,$ and $\mathcal{F}_i\cup\mathcal{F}_j$ is also a filter (either $\mathcal{F}_i$ or $\mathcal{F}_j$),  which implies that $\mathcal{X}\cap\mathcal{Y}\in\Gamma$.

		\item[(c)] Let $\mathcal{X}\in\Gamma$ and let $\mathcal{Y}\in\mathcal{P}(\mathbb N)$ 
		Suppose $\mathcal{X}\subseteq\mathcal{Y},$ then $\mathcal{X}\in\mathcal{F},$ where $\mathcal{F}$ is in the union of $\Gamma$.
		Since $\mathcal{F}$ is a free filter and $\mathcal{X}\subseteq\mathcal{Y}$ then $\mathcal{Y}\in\mathcal{F}.$ Thus $\mathcal{Y}$ is in $\Gamma$.

		\item[(d)] Suppose to the contrary that
		\[
          		          m\in\bigcap_{\mathcal{X}\in\Gamma}\mathcal{X}
		\]
		for some \text{m}$\in\mathbb N$. Then $(\forall x\in\Gamma)(m\in\mathbb N)$, but this implies that for some $\mathcal{F}$ in $\mathcal{C}$
		\[
         		           m\in\bigcap_{\mathcal{X}\in\mathcal{F}}\mathcal{X}
		\]
		which contradicts the fact that $\mathcal{F}$  is a free filter.
	\end{description}

	Thus  $\Gamma\in\mathcal{S}$. Then for any chain in $\mathcal{S}$ there exists an upper bound $\Gamma$.
	Utilizing \text{Zorn's Lemma} we know that $\mathcal{S}$ contains a maximal element, say $\mathcal{U}$. By construction we know that $\mathcal{F}_0\subseteq\mathcal{U}$, thus proving that every free filter can be extended to a free ultrafilter.
\end{proof}

\section{Characterization of the Ultrafilter}

\begin{lemma}\label{Thm: union}
	Let $A_{1}, A_{2},\dots, A_{n}\in\mathcal{P(\mathbb N)}$ such that $A_{1}\cup A_{2}\cup\dots\cup A_{n}\in\mathcal{U},$ where $\mathcal{U}$ is an ultrafilter on $\mathbb N$. Then $A_{i}\in\mathcal{U}$ for at least one $i$. In addition, if the sets are mutually 		disjoint, then $A_{i}\in\mathcal{U}$ for exactly one $i$.
\end{lemma}

\begin{proof}
	Let $A_{1}\cup A_{2}\in\mathcal{U}$. Suppose (to the contrary) that neither $A_{1}\in\mathcal{U}$ or $A_{2}\in\mathcal{U}$. Observe that 
	\[
		\mathcal{M}=\{X\in\mathcal{P}(\mathbb N): A_{1}\cup X\in\mathcal{U}\}
	\]
	is also a filter on $\mathbb N$. Notice that $\mathcal{U}\subseteq\mathcal{M}$ by property (c) in  Definition~\ref{Def: Frechet}. Also,
	$\mathcal{U}\subsetneq\mathcal{M}$ because $A_{2}\in\mathcal{M}\setminus\mathcal{U}$, contradicting the maximality of $\mathcal{U}$. Finally, if $A_{1}\cap A_{2} =\varnothing$ and  $A_{1},A_{2}\in\mathcal{U}$ then this implies that   		 
            $\varnothing\in\mathcal{U}$, a contradiction. The generalization to n $\ge 2$ follows simply by induction.
\end{proof}

\begin{theorem}\label{Thm: Char}
	Let $\mathcal{F}$ be a filter on $\mathbb N$. Then the following are equivalent:
	\begin{description}
		\item[(i)] $\mathcal{F}$ is maximal (ultrafilter).
		\item[(ii)]($\forall A\in\mathcal{P}(\mathbb N))$ ($A\in\mathcal{F}\text{ or } \mathbb N\setminus A \in\mathcal{F}$).
	\end{description}
\end{theorem}

\begin{proof}
\hspace{.5in}
	\begin{description}

		\item[(i)$\Rightarrow$(ii)]
		Let  $A\in\mathcal{P}(\mathbb N)$. Then Lemma \ref{Thm: union} holds since $A\cup(\mathbb N\setminus A)=\mathbb N\in\mathcal{U}$ and $A\cap(\mathbb N\setminus A)=\varnothing$.
		\item[(ii)$\Rightarrow$(i)]
		Suppose (to the contrary) that $\mathcal{F}$ is not maximal. Then $\mathcal{F}$  is properly contained in some free filter $\mathcal{M}$. 
		The complement of $\mathcal{M}\setminus\mathcal{F}$ will consist of some set $\mathcal{B}\in\mathcal{M}$ where
 		$\mathcal{B}\not\in\mathcal{F}$. By (ii) we know that $\mathbb N \setminus \mathcal{B}\in\mathcal{F}$. Since $\mathcal{F}\subset\mathcal{M}$ 
		then this implies that both $\mathcal{B}$ and $\mathbb N\setminus\mathcal{B}$ are in $\mathcal{M}$. Recall that $\mathcal{M}$ is a filter and 
		is therefore closed under intersections. Thus $\mathcal{B}\cap(\mathbb N\setminus\mathcal{B})=\varnothing\in\mathcal{M}$, contradicting the fact that $\mathcal{M}$ is a filter.
	\end{description}
\end{proof}

\begin{corollary}\label{Cor: FrechetNotMaximal}
	 The Fr\'{e}chet filter, $\mathcal F_r,$ (Example~\ref{Ex: Filters}) is not an ultrafilter.
\end{corollary}

\begin{proof} 
	Let $\mathbb E$ and $\mathbb O$ denote the sets of the even and odd numbers in $\mathbb N$, respectively. 
            It is clear that $\mathbb E\cap\mathbb O=\varnothing$ and $\mathbb E\cup\mathbb O=\mathbb N\in\mathcal F_r$, but neither $\mathbb E$ nor 
	$\mathbb O$ belongs to $\mathcal F_r$.
\end{proof}

\begin{theorem}
	An ultrafilter $\mathcal{U}$ on $\mathbb N$ is free if and only if $\mathcal{F}_r\subset\mathcal{U}$, where $\mathcal{F}_r$ is the
	 Fr\'{e}chet filter on $\mathbb N$.
\end{theorem}

\begin{proof}
\hspace{.5in}
	\begin{description}
		\item[($\Rightarrow$)] Let $\mathcal{U}$ be a free ultrafilter on $\mathbb N$. Suppose (to the contrary) that $\mathcal{F}_r\not\subset\mathcal{U}$.
 		This implies that there exists an $S\in\mathcal{F}_r$ such that $S\not\in\mathcal{U}$. By Theorem, \ref{Thm: Char}, if 
		$S\not\in\mathcal{U}$ then the finite set, $\mathbb N\setminus S\in\mathcal{U}$, contradicting the fact that $\mathcal{U}$ is a free filter.
		\item[($\Leftarrow$)] Let $\mathcal{U}$ be an ultrafilter on $\mathbb N$ such that $\mathcal{F}_{r}\subset\mathcal{U}$. Suppose (to the contrary) that $\mathcal{U}$ is not free. Then there exists an $a\in\mathbb N$ such that
			\[
                                                 a\in\bigcap_{A\in\mathcal{U}}A\subseteq\bigcap_{A\in\mathcal{F}_{r}}A=\varnothing,
			\]
                        a contradiction.
	\end{description}
\end{proof}

This result allows us to easily check if a filter $\mathcal{U}$ is indeed a free ultrafilter, since if $A\in\mathcal{F}_r$ then $A\in\mathcal{U}$.

\chapter{The Fr\'{e}chet Filter in Real Analysis}

\section{Fr\'{e}chet Filter}
In this section we define what it means for a filter to be Fr\'{e}chet as well as the properties that characterize the  Fr\'{e}chet filter.

\begin{definition}[Fr\'{e}chet Filter]\label{Def: Frechet}
Let $\mathcal{F}_r$ denote the set of all cofinite subsets of $\mathbb N$, meaning
\[
\mathcal{F}_r = \{S\in\mathcal{P}(\mathbb N):\mathbb N\setminus S \text{ is finite}\}.
\]
We call $\mathcal{F}_r$ the Fr\'echet filter on $\mathbb N$.
\end{definition}

The following lemmas will serve to highlight the key properties of the Fr\'echet filter that shall be used in following sections.

\begin{lemma}\label{Lem1: Frechet}
Let $A\subseteq\mathbb N$. Then $A\in\mathcal{F}_r$ if and only if there exists $\nu\in\mathbb N$ such that $\{\nu,\nu +1, \nu +2,\dots\}\subseteq A$. In particular, $\{\nu,\nu +1, \nu +2,\dots\}\in\mathcal{F}_r$ for any $\nu\in\mathbb N$.
\end{lemma}

\begin{proof}
\begin{description}
\item[($\Rightarrow$)] Let $A\in\mathcal{F}_r$. Then $\mathbb N\setminus A$ is finite and has the form
\[
\mathbb N\setminus A = \{a_1,a_2,\dots , a_m\}
\]
for some $m\in\mathbb N$ and $a_i\in\mathbb N$. The latter implies that $\mathbb N\setminus A\subseteq \{1,2,\dots , a_m\}$. Taking the complement once more, we have
\[
A\supseteq\{a_{m}+1,a_{m}+2,a_{m}+3,\dots\}.
\]
Thus $\{\nu,\nu+1,\nu+2,\dots\}\subseteq A$ holds for $\nu=a_m+1$.


\item[($\Leftarrow$)]  Let $A\in\mathcal{P}(\mathbb N)$ such that $\{\nu,\nu+1,\dots\}\subseteq A$, where $\nu\in\mathbb N$. Taking the complement, we have
\[
\mathbb N\setminus A\subseteq\{1,2,\dots,\nu-1\}.
\]
Thus $\mathbb N\setminus A$ is finite.

\end{description}
\end{proof}

\begin{lemma}\label{Lem2: Frechet}
$\mathcal{F}_r$ is a free filter in the sense that:
\begin{description}
\item[(i)] $\varnothing\not\in\mathcal{F}_r$.
\item[(ii)] $\mathcal{F}_r$ is closed under finitely many intersections.
\item[(iii)] Let $A\in\mathcal{F}_r$ and $B\subseteq\mathbb N$. Then $A\subseteq B$ implies that $B\in\mathcal{F}_r$.
\item[(iv)] $\bigcap_{A\in\mathcal{F}_r} A = \varnothing$.	
\end{description}
\end{lemma}

\begin{proof}
\hspace{.5in}
\begin{description}
\item[(i)] Clearly, $\varnothing\not\in\mathcal{F}_r$ since $\mathbb N\setminus\varnothing = \mathbb N$ is infinite.
\item[(ii)] Let $A,B\in\mathcal{F}_r$ and define $A^{c}=\mathbb N\setminus A$ to be the complement of $A$ with respect to $\mathbb N$. Then $A\text{ and }$ B are cofinite sets. Thus
\[
 (A\cap B)^c = A^c\cup B^{c},
\]
which is clearly a finite set. Therefore $A\cap B\in\mathcal{F}_r$.
\item[(iii)] Let $A\in\mathcal{F}_r$ and $B\subseteq\mathbb N$ such that $A\subseteq B$. By Lemma \ref{Lem1: Frechet}
we know that $\{\mu, \mu +1, \mu +2,\dots\}\subseteq A$. Since $A\subseteq B$ then $\{\mu, \mu +1, \mu +2,\dots\}\subseteq B$, which implies that $B\in\mathcal{F}_r$.
\item[(iv)] Suppose, to the contrary, that 
\[
a\in\bigcap_{A\in\mathcal{F}_r} A,
\]
Then $a\in\{a+1,a+2,\dots\}$, a contradiction since $\{a+1,a+2,\dots\}\in\mathcal{F}_r$ by Lemma \ref{Lem1: Frechet}.
\end{description}
\end{proof}

\section{ Reduction in the Number of Quantifiers}

In this section, we demonstrate how characterizing sequence convergence in terms of the Fr\'echet filter leads to a reduction in the number of quantifiers from three to one.\\

Let $(a_n)$ be a sequence in $\mathbb R$, $L\in\mathbb R$ and $\varepsilon\in\mathbb R_+$. We denote:
\begin{equation}\label{E: The Set S}
S_\varepsilon=\{n\in\mathbb N: |a_n-L|<\varepsilon\}.
\end{equation}

\begin{theorem}\label{Thm: Limit}
Let $(a_n)$ be a sequence in $\mathbb R$ and let $L\in\mathbb R$ and $\varepsilon\in\mathbb R_+$. Then the following are equivalent:
\begin{description}
\item[(i)] $lim_{n\rightarrow\infty}a_n = L$ in the sense that
\[
 (\forall\varepsilon\in\mathbb R_+)(\exists\nu\in\mathbb N)(\forall n\in\mathbb N)(n\geq\nu \Rightarrow |a_n-L|<\varepsilon).
\]
\item[(ii)]  $(\forall \varepsilon\in\mathbb R_+)(S_\varepsilon\in\mathcal{F}_r)$, where $S_\varepsilon$ is the set (\ref{E: The Set S}).\\
\end{description}
\end{theorem}

\begin{proof}
\hspace{.5in}
\begin{description}
\item[(i)$\Rightarrow$(ii)]: Let $\epsilon\in\mathbb R_+$ be chosen arbitrarily. By the assumption (i), we have 
\[
(\exists\nu\in\mathbb N) (\forall n\in\mathbb N) (n\ge\nu\Rightarrow |a_n - L|<\epsilon).
\]
Thus $\{\nu,\nu+1,\nu+2\dots\}\subseteq S_\epsilon$. The latter implies that $S_\epsilon \in\mathcal{F}_r$ by Lemma \ref{Lem1: Frechet}.

\item[(ii)$\Rightarrow$(i)]: Let $\epsilon\in\mathbb R_+$ be chosen arbitrarily. Then $S_\epsilon \in\mathcal{F}_r$ implies that 
$\{\nu,\nu+1,\nu+2,\dots\}\subseteq S_\epsilon$, for some $\nu\in\mathbb N$ by Lemma \ref{Lem1: Frechet}. We interpret this as
\[
\{n\in\mathbb N : n\ge\nu\}\subseteq\{n\in\mathbb N: |a_n - L|<\epsilon\}.
\]
The latter being equivalent to 
\[
(\forall\epsilon\in\mathbb R_+) (\exists\nu\in\mathbb N) (\forall n\in\mathbb N) (n\ge\nu \Rightarrow |a_n - L|<\epsilon),
\]
as required.
\end{description}
\end{proof}

\begin{corollary}[Negation]\label{C: Negation} Under the assumption of the above theorem, the following are equivalent:
\begin{description}
\item[(i)] $lim_{n\rightarrow\infty}a_n = L$ is false, that is
\[
 (\exists\varepsilon\in\mathbb R_+)(\forall\nu\in\mathbb N)(\exists n\in\mathbb N)(n\geq\nu \text{ and } |a_n-L|\geq\varepsilon).
\]
\item[(ii)]  $(\exists \varepsilon\in\mathbb R_+)(S_\varepsilon\notin\mathcal{F}_r)$, where $S_\varepsilon$ is the set (\ref{E: The Set S}).
\end{description}
\end{corollary}

\section{Fr\'echet filter in Real Analysis}
In this section we show how the characterization of the limit in terms of $\mathcal{F}_{r}$ works in practice.

\begin{theorem}[Squeeze Theorem]\label{Thm: sqz}
Let $(a_n), (b_n), (x_n)\in\mathbb R^\mathbb N$. Let $a_{n}\leq x_{n}\leq b_{n}$ hold for all sufficiently large $n$ and let $lim_{n\rightarrow\infty} a_n = \lim_{n\rightarrow\infty} b_n= L$. Then $\lim_{n\rightarrow\infty} x_n = L$.
\end{theorem}

\begin{proof}
Let $\epsilon\in\mathbb R_+$. We define the following sets
\begin{description}
\item $X=\{n\in\mathbb N : a_n\leq x_n \leq b_n\}$.
\item $A_{\epsilon}=\{n\in\mathbb N : L -\epsilon < a_n < L+\epsilon\}$.
\item $B_{\epsilon}=\{n\in\mathbb N : L - \epsilon < b_n < L+\epsilon\}$.
\end{description}
By assumption we know that $X, A_{\epsilon},\text{ and } B_{\epsilon}$ are all members of $\mathcal{F}_r$. Since the Fr\'echet filter is closed under finite intersections, then $X\cap A_{\epsilon}\cap B_{\epsilon}\in\mathcal{F}_r$. However, 
\[
X\cap A_{\epsilon} \cap B_{\epsilon} = \{n\in\mathbb N : L-\epsilon < x_n < L+\epsilon\}.
\]
Thus, by Theorem \ref{Thm: Limit}, we see that $\lim_{n\rightarrow\infty} x_n = L$.
\end{proof}

\begin{theorem}\label{Thm:com}
The limit operation preserves order in the sense that \\
 if $a_n \leq b_n$ for all sufficiently large $n$ and $\lim_{n\rightarrow\infty} a_n =a ,\lim_{n\rightarrow\infty} b_n = b$,
then $a\leq b$.
\end{theorem}

\begin{proof}
Suppose, to the contrary, that $\lim_{n\rightarrow\infty} a_n > \lim_{n\rightarrow\infty} b_n$. Let $\epsilon = \frac{a-b}{2}$. We define the following sets:

\begin{description}
\item $A:=\{n\in\mathbb N : |a_n - a| < \frac{a-b}{2}\}\in\mathcal{F}_r$,
\item $B:=\{n\in\mathbb N : |b_n - b| < \frac{a-b}{2}\}\in\mathcal{F}_r$.
\item $C:=\{n\in\mathbb N : b_n < a_n\}$.
\end{description}
The sets $A$ and $B$ can be rewritten as
\begin{description}
\item $A=\{n\in\mathbb N: \frac{a+b}{2} < a_n < \frac{3a - b}{2}\}$.
\item $B=\{n\in\mathbb N: \frac{2b-a}{2} < b_n < \frac{a+b}{2}\}$.
\end{description}
Next, we observe that $A\cap B\subseteq C$. Indeed, $n\in A\cap B$ implies that $b_n < \frac{a+b}{2} < a_n$ or in other words $n\in C$. Thus $C\in\mathcal{F}_r$, but this implies that $\mathbb N\setminus C$ is finite, contradicting the fact that $a_n \leq b_n$ almost everywhere.
\end{proof}

The above examples demonstrate how the Fr\'echet filter slightly simplifies some proofs in analysis. The main advantage, in the opinion of the author, is that it does away with limit arguments, which typically are a source of confusion among beginning students. Instead, the proofs can be easily completed using the basics of set theory.

\section{Remarks Regarding the Fr\'echet Filter}
As we have shown, the  Fr\'{e}chet filter can be used to reduce the number of quantifiers needed in the real analysis. This result leads us to wonder if it would be beneficial to construct a number system where the elements are imbued with the properties of the  Fr\'{e}chet filter. This new system would be constructed in a manner similar to Cauchy's construction of the real numbers from rational sequences. The elements in this new system would be equivalence classes of real numbered sequences, which take into account sequence convergence (divergence) as well as the rate of convergence (divergence). Ideally, the resulting system will contain elements that can be used to characterize convergence in such a manner that we can do away with the limits of standard analysis or the set constructions from the  Fr\'{e}chet approach.

Let us consider the factor ring
\[
\tilde{\mathbb R}^\mathbb N = \mathbb R ^\mathbb N/ \sim_{\mathcal{F}_{r}}
\]
where $ \sim_{\mathcal{F}_{r}}$ is the equivalence relation defined by
\[
 (a_n){ \sim_{\mathcal{F}_{r}}}(b_n)  \text{ if and only if } \{n:a_n = b_n\}\in\mathcal{F}_r.
\]
 This is no different to saying that $(a_n)$ is equivalent to $(b_n)$ if and only if $a_n = b_n$ for all sufficiently large n. Thus the elements in our new system are equivalence classes of real sequences, denoted by $\bra a_{n} \ket$.
We now define the relevant operations and order of our new system.

\begin{definition}\label{Frechet: Order}
$(a_n) \le (b_n)$ if and only if $\{n: a_n \le b_n\}\in\mathcal{F}_r$
\end{definition}

\begin{definition}\label{Frechet: AM}
Let ${x}\text{ and }{y}$ be elements in $\tilde{\mathbb R}^\mathbb N$ such that ${x}=\bra x_n \ket$ and 
${y}=\bra y_n \ket$. Then we have the following operations:
\begin{enumerate}
\item ${x+y}$ = $\bra x_n + y_n \ket$.
\item ${x\cdot y}$ = $\bra x_n \cdot y_n \ket$.
\item ${x\le y}$ if and only if $(x_n) \le (y_n)$.
\end{enumerate}
\end{definition}

It is easy to show that the above operations are well defined.\\

Does our new structure live up to our lofty ambitions? The sad fact is no. The new construction we have devised is no better than $\mathbb R^\mathbb N$. At most,  $\tilde{\mathbb R}^\mathbb N$ is a partially ordered ring with zero divisors. For example,

\[
\bra 1,0,1,\dots\ket\cdot \bra 0,1,0,\dots\ket = \bra 0,0,0,0,\dots\ket.
\]

Even worse, this new construction does not posses the law of the excluded middle, leading to elements which cannot be ordered relative to one another. For example, neither $\bra 1,0,1,\dots\ket\le\bra 0,1,0,\dots\ket$, nor 
$\bra 0,1,0,\dots\ket \le\bra 1,0,1,\dots\ket$ are true statements, as demonstrated by Corollary \ref{Cor: FrechetNotMaximal}.

Clearly, the Fr\'echet construction is inferior to $\mathbb R$ and cannot be applied to real analysis. However, not all is lost. Indeed, we can strengthen the Fr\'echet filter by extending it to a free ultrafilter, as shown in section \ref{Sec: UltraExistence}. This extension allows us to transform  $\tilde{\mathbb R}^\mathbb N$ into a totally ordered field known as the non-standard real numbers, denoted by  $^*\mathbb R$. The next chapter shall show that this new structure will be an extension of $\mathbb R$ and shall posses several unique properties that greatly simplify our work in the real analysis.

\chapter{Non-standard Analysis}

\section{Construction of the Hyperreals $^*\mathbb R$}\label{Sec: Const}
The construction of $^*\mathbb R$ is reminiscent of the construction of the reals from the rationals by means of equivalence classes of Cauchy Sequences.
To begin, we start with $\mathbb R^\mathbb N$, which is the set of sequences of real numbers. Each member of $\mathbb R ^\mathbb N$ has the form
\[
X=(x_{n}:n\in\mathbb N)
\]
or for simplicity, $(x_{n})$. $\mathbb R^\mathbb N$ is considered to be a commutative ring with unity under the usual operations of pointwise addition and multiplication. Furthermore, $\mathbb R ^\mathbb N$ is partially ordered under the following relation
\[
(x_{n})\le (y_{n})\text{ if and only if } \{n:x_{n}\le y_{n}\}\in\mathcal{F}_{r}.
\]
Despite it's rich structural properties, $\mathbb R ^\mathbb N$ fails to be a totally ordered field due to the existence of zero divisors
\[
(1,0,1,0,\ldots) \cdot (0,1,0,1,\ldots) = (0,0,0,\ldots),
\]
as well as the existence of elements that cannot be ordered.
To rectify this situation we shall define an equivalence relation on $\mathbb R ^\mathbb N$, creating a new set $^*\mathbb R$, as well as defining new operations which will make $^*\mathbb R$ into a linearly ordered field.\\

Let $\mathcal{U}$ be a free ultrafilter on $\mathbb N$. We define a relation, $\equiv$, on $\mathbb R ^\mathbb N$ as follows.

\begin{definition}
If $X=(x_{n})$ and $Y=(y_{n})$ are in $\mathbb R ^\mathbb N$, then
\[
 (x_{n}) \equiv (y_{n}) \text{ if and only if } \{n\in\mathbb N: x_{n}=y_{n}\}\in\mathcal{U}
\]
\end{definition}

\begin{lemma}
The relation ,$\equiv$, is an equivalence relation on $\mathbb R ^\mathbb N$.
\end{lemma}

\begin{proof}
\begin{description}
\item[Reflexive:] Let X = $(x_{n})$ $\in\mathbb R ^\mathbb N$. Then $\{n\in\mathbb N:x_{n}=x_{n}\}=\mathbb N\in\mathcal{U}.$ Thus $(x_{n})$ $\equiv$ $(x_{n})$.

\item[Symmetric:] Let X = $(x_{n})$ and let Y = $(y_{n})$ be elements of $\mathbb R ^\mathbb N$ such that $(x_{n})\equiv(y_{n})$, which implies that $\{n\in\mathbb N:x_{n}=y_{n}\}\in\mathcal{U}.$ By the symmetry of = on 
$\mathbb R$ we see that $\{n\in\mathbb N:x_{n}=y_{n}\}=\{n\in\mathbb N:y_{n}=x_{n}\}$. Thus $(y_{n})\equiv(x_{n})$.

\item[Transitive:] Let X = $(x_{n})$, Y = $(y_{n})$, and Z = $(z_{n})$ such that $(x_{n})\equiv(y_{n})$ and 
$(y_{n})\equiv(z_{n})$. Let $\mathcal{S}_{1}=\{n\in\mathbb N:x_{n}=y_{n}\}$ and let $\mathcal{S}_{2}=\{n\in\mathbb N:y_{n}=z_{n}\}$, both of which are members of $\mathcal{U}$. Since $\mathcal{U}$ is closed under intersections
\[
\mathcal{S}_{1}\cap\mathcal{S}_{2}=\{n\in\mathbb N:x_{n}=y_{n}\text{ and } y_{n}=z_{n}\}\in\mathcal{U}.
\]
Then $\mathcal{S}_{1}\cap\mathcal{S}_{2}\subseteq\{n\in\mathbb N:x_{n}=z_{n}\}$ implies that $\{n\in\mathbb N:x_{n}=z_{n}\}\in\mathcal{U}$. Thus $(x_{n})\equiv(z_{n})$.
\end{description}
\end{proof}

This method is known as the ultraproduct construction of the set of  nonstandard or hyperreal numbers, which are denoted by $^*\mathbb R$. We now introduce the operations on $^* \mathbb R$.

\begin{definition}
Let $x$ and $y$ be elements in $^*\mathbb R$ such that $x$ = $\bra x_{n}\ket$ and \newline $y$ = $\bra y_{n}\ket$. Then we have the following operations:
\begin{enumerate}
\item ${x+y}$ = $\bra x_{n}+ y_{n}\ket$.
\item ${x\cdot y}$ = $\bra x_{n}\cdot y_{n} \ket$.
\item ${x<y}$ if and only if $\{n\in\mathbb N: x_{n}<y_{n}\}\in\mathcal{U}$ and ${x\le y}$ if and only if ${x<y}$ or ${x=y}$.
\end{enumerate}
\end{definition}
%
%
%
%
%

\begin{theorem}
$^*\mathbb R$ is a linearly ordered field.
\end{theorem}

\begin{proof}
$^*\mathbb R$ is already a partially ordered commutative ring with unity. To show that $^*\mathbb R$ is a field, suppose that $X=\bra x_{n}\ket \in { ^*\mathbb R}$ such that ${\bra x_{n}\ket}\not = 0$. Then $\{n\in\mathbb N : x_{n}=0\}\not\in\mathcal{U}$ and so 
$\{n\in\mathbb N : x_{n}\not = 0\}\in\mathcal{U}$ by Theorem \ref{Thm: Char}. We define $X^{-1}=\bra\bar{x}^{-1}_{n}\ket$, where $\bar{x}^{-1}_{n}=x^{-1}_{n}$ if $x_{n}\not = 0$ and $\bar{x}^{-1}_{n} = 0$ if $x_{n}=0$. Recall that
 $X\cdot X^{-1}=1$ if and only if $\{n\in\mathbb N: x_{n}\cdot\bar{x}^{-1}_{n}=1\}\in\mathcal{U}$. This relation holds since 
$\{n\in\mathbb N : x_{n}\not = 0\}\subseteq \{n\in\mathbb N: x_{n}\cdot\bar{x}^{-1}_{n}=1\}$.\\

To show that $^*\mathbb R$ is linearly ordered. Suppose that $\bra x_{n}\ket,\bra y_{n}\ket \in {^*\mathbb R}$ and denote 
$A=\{n\in\mathbb N: x_{n}<y_{n}\}$, $B=\{n\in\mathbb N : x_{n}=y_{n}\}$, and $C=\{n\in\mathbb N : x_{n}>y_{n}\}$. Since 
$A\cup B\cup C =\mathbb N\in\mathcal{U}$, then by Lemma \ref{Thm: union} exactly one of the sets are in the ultrafilter $\mathcal{U}.$ Thus exactly one of the following relations hold:
\[
\bra x_{n}\ket< \bra y_{n}\ket\text{,             }\bra x_{n}\ket=\bra y_{n}\ket\text{ or             }\bra x_{n}\ket>\bra y_{n}\ket.
\]
Therefore  $^*\mathbb R$ is a linearly ordered field under our newly defined operations.
\end{proof}

We conclude this section by showing that $\mathbb R$ can be imbedded isomorphically as a linearly ordered subfield of $^*\mathbb R$ by the following mapping.

\begin{definition}
We define $*:\mathbb R\rightarrow {^*\mathbb R}$ to be a mapping such that $*(x)={^*x}$, where ${^*x}=\bra x,x,x,\dots\ket\in{^*\mathbb R}$.
\end{definition}

\begin{theorem}
The mapping $*$ is an order preserving isomorphism of $\mathbb R$ into a subfield of ${^*\mathbb R}$.
\end{theorem}


\section{Finite, Infinitesimal, and Infinitely Large Numbers}
We classify the elements of the hyperreals, how they behave under the operations defined on $^*\mathbb R$, 
and how they relate to $\mathbb R$. We conclude this section by defining the standard part mapping, 
which will serve an important role in our treatment of the nonstandard analysis.

\begin{definition}[Classification]
 Let ${{x}}\in{^*\mathbb R}$
	\begin{description}
		\item[(a)] ${{x}}$ is ${\bf{ infinitesimal}}$ if $\mid{{x}}\mid < \epsilon$ for all $\epsilon\in\mathbb R_{+}$.  		We denote the set of all infinitesimals by $\mathcal{I}({^*\mathbb R})$.
		\item[(b)] ${{x}}$ is ${\bf{finite}}$ if ${\mid{{x}}\mid \le\epsilon}$ for some $\epsilon\in\mathbb R_{+}$.			            We denote the set of all finite numbers by $\mathcal{F}({^*\mathbb R})$.
		\item[(c)] ${{x}}$ is {\bf{infinitely large}} if $\mid{{x}}\mid > \epsilon$ for all $\epsilon\in\mathbb R_{+}$.  		We denote the set of all infinitely large numbers by $\mathcal{L}({^*\mathbb R})$.
	\end{description}
\end{definition}

\begin{example}[Infinitesimal]
 Let $\epsilon\in\mathbb R_{+}$ be arbitrary. Then $\bra{\frac{1}{n}}\ket$ is a positive infinitesimal or in other words
$0< \bra{\frac{1}{n}}\ket <\epsilon$. Clearly $\bra{\frac{1}{n}}\ket > 0$ since $\{n\in\mathbb N:\frac{1}{n}>0\}=\mathbb N\in\mathcal{U}$.
 Finally, $\bra{\frac{1}{n}}\ket < \epsilon$, where $\epsilon = \bra \epsilon,\epsilon,\epsilon\dots\ket$, because $\frac{1}
{n}<\epsilon$ implies that $\ n >\frac{1}{\epsilon}$. Let $\nu=min\{n\in\mathbb N:n>\frac{1}{\epsilon}\}$.
 Then $\{n:\frac{1}{n}<\epsilon\}=\{\nu,\nu+1,\nu+2,\dots\}\in\mathcal{U}$. Therefore $\bra{\frac{1}{n}}\ket$ is an infinitesimal.
\end{example}

\begin{example}[Finite]
It is clear that all real numbers are finite in $^*\mathbb R$. Here is an example for a finite, but standard number: $\bra r+\frac{1}{n}\ket
= r+{\bra\frac{1}{n}\ket}\in{\mathcal{F}({^*\mathbb R})\setminus\mathbb R}$.
\end{example}

\begin{example}[Infinitely Large]
$\bra n \ket$ is a positive, infinitely large number. Let $\epsilon\in\mathbb R_{+}$ be arbitrary and let $\nu=min\{n\in\mathbb N: \epsilon < n\}$. Then $\{n\in\mathbb N:\epsilon < n\} = \{\nu,\nu+1,\nu+2,\dots\}\in\mathcal{U}$.
 Thus $\bra n \ket > \epsilon$ and therefore $\bra n \ket$  is infinitely large. 
\end{example}

\begin{remark}
Observe that if $(a_{n})$ is any real-valued sequence converging to zero, then $\bra a_{n}\ket$ is an infinitesimal in ${^*\mathbb R}$.
Alternatively, if $(a_{n})$ is any real-valued sequence diverging to infinity, then $\bra a_{n}\ket$ is infinitely large in  ${^*\mathbb R}$.
\end{remark}
 The existence of these elements show that $^*\mathbb R$ is a proper extension of $\mathbb R$. We conclude this section by demonstrating that
 $\mathcal{F}({^*\mathbb R}) /  \mathcal{I}({^*\mathbb R})$ is isomorphic to a subfield of  $\mathbb R$. Indeed, by using the fact that $\mathbb R$ is a complete field,
 we shall show that $\mathcal{F}({^*\mathbb R}) /  \mathcal{I}({^*\mathbb R})$ is in fact isomorphic to $\mathbb R$.

\begin{theorem}
The set of finite numbers,  $\mathcal{F}({^*\mathbb R})$, forms a  subring of $^*\mathbb R$.
\end{theorem}

\begin{proof}
Since $\mathcal{F}({^*\mathbb R})$ inherits the addition and multiplication from the field $^*\mathbb R$, we must only show that it  is closed under these operations.

Let $ {a,b}\in\mathcal{F}({^*\mathbb R})$. Then there exists ${\epsilon_{1},\epsilon_{2}}\in\mathbb R_{+}$ such that 
$ {\mid a \mid < \epsilon_{1}}$ and ${ \mid b \mid < \epsilon_{2} }$. Thus ${\mid a + b \mid} \le {\mid a \mid} + {\mid b \mid} < \epsilon_{1} + \epsilon_{2}$ and therefore ${a+b}\in\mathcal{F}(^*\mathbb R)$. Similarly, ${\mid {a\cdot b} \mid} < {\epsilon_{1} \cdot \epsilon_{2}}$, and therefore ${a \cdot b}\in\mathcal{F}(^*\mathbb R)$.
\end{proof}

\begin{theorem}
$\mathcal{I}({^*\mathbb R})$ is a maximal ideal of $\mathcal{F}({^*\mathbb R})$. Consequently, ${\mathcal{F}({^*\mathbb R})} / {\mathcal{I}({^*\mathbb R})}$ is a subfield of $\mathbb R$.
\end{theorem}

\begin{proof}
First, we must show that $\mathcal{I}({^*\mathbb R})$ is an ideal of $\mathcal{F}({^*\mathbb R})$. Let $a\in\mathcal{F}({^*\mathbb R})$ and 
$b\in\mathcal{I}({^*\mathbb R})$. Then there exists an $\epsilon_{1}\in\mathbb R_{+}$ such that ${\mid a \mid} < {\epsilon_{1}}$.
 Furthermore, for arbitrary $\epsilon_{2}\in\mathbb R_{+}$ we have ${\mid b \mid} < {\frac{\epsilon_{2}}{\epsilon_{1}}}$. Then 
${\mid a \cdot b \mid} < \epsilon_{2}$, which implies ${a \cdot b}\in\mathcal{I}({^*\mathbb R})$ and that $\mathcal{I}({^*\mathbb R})$ is an ideal of $\mathcal{F}({^*\mathbb R})$.\\

To show that $\mathcal{I}({^*\mathbb R})$ is maximal, suppose (to the contrary) that there exists an ideal $\mathcal{J}$ of  $\mathcal{F}({^*\mathbb R})$ such that
\[
 \mathcal{I}({^*\mathbb R})\subsetneq\mathcal{J}\subsetneq\mathcal{F}({^*\mathbb R})
\]
Let $\alpha\in{\mathcal{J}\setminus\mathcal{I}({^*\mathbb R})}$. Since $\alpha\not=0$, then its inverse, $\alpha^{-1}$, exists in the field $^*\mathbb R$. 
It remains to show that $\alpha^{-1}\in\mathcal{F}({^*\mathbb R})$. Indeed, there exists $\epsilon\in\mathbb 
R_{+}$ such that $\epsilon \le {\mid\alpha\mid}$, since $\alpha\not\in\mathcal{I}({^*\mathbb R})$. Thus ${\mid\frac{1}{\alpha}\mid} \le \frac{1}{\epsilon}$, 
which implies that $\alpha^{-1} \in\mathcal{F}({^*\mathbb R})$. So $1=\alpha\cdot\alpha^{-1}\in\mathcal{J}$, 
implying that $\mathcal{J} = \mathcal{F}({^*\mathbb R})$, a contradiction.\\

We conclude the proof by showing that ${\mathcal{F}({^*\mathbb R})} / {\mathcal{I}({^*\mathbb R})}$ is a subfield of $\mathbb R$. Clearly this
factor ring is a field since $\mathcal{I}({^*\mathbb R})$ is a maximal ideal. Additionally, since $\mathcal{F}({^*\mathbb R})$ is an Archimedean ring,
 our field must also be Archimedean. It is well known in mathematics that every ordered Archimedean field is isomorphic to a 
subfield of the real numbers. Therefore, ${\mathcal{F}({^*\mathbb R})} / {\mathcal{I}({^*\mathbb R})}$ is isomorphic to a subfield of $\mathbb R$.
\end{proof}

To show that ${\mathcal{F}({^*\mathbb R})} / {\mathcal{I}({^*\mathbb R})}$ is isomorphic to $\mathbb R$, we first must characterize the elements in $\mathcal{F}({^*\mathbb R})$.

\begin{theorem}[Characterization of Finite Numbers]\label{thm: finite}
Every $x\in\mathcal{F}({^*\mathbb R})$ can be written uniquely as the sum
\[
x=r+h,
\]
where $r\in\mathbb R$ and $h\in\mathcal{I}(^*\mathbb R)$.
\end{theorem}

\begin{proof} Let $r=\sup\{a\in{\mathbb R}:a<x\}$. The existence of $r$ is guaranteed by the completeness of $\mathbb R$ since the set $\{a\in{\mathbb R}:a<x\}$ is bounded from above. It remains to show that $x-r$ is an infinitesimal. Suppose (to the contrary) that $x-r$ is not an infinitesimal. Then there exists an $\epsilon\in\mathbb R_{+}$ such that $\epsilon\le{\mid x-r\mid}$. If $x-r>0$, then  this implies that $\epsilon+r<x$, contradicting our choice of r. Alternatively, if $x-r<0$ then $x<r-\epsilon$, which also contradicts our choice of r as the supremum. Therefore $x-r$ is an infinitesimal.\\

We now show that this sum is unique. Let $x\in\mathcal{F}(^*\mathbb R)$ such that 
\[
x=r_{1}+h_{1} \text{ and } x=r_{2} + h_{2},
\]
 where $r_{1},r_{2}\in\mathbb R$ and $h_{1},h_{2}\in\mathcal{I}(^*\mathbb R)$. Rearranging the terms, we have
\[
r_{1}-r_{2} = h_{2} - h_{1}.
\]
Observe that the left hand side is a real number while the right hand side is an infinitesimal. The only element that is simultaneously real and infinitesimal is 0. Thus $r_{1}=r_{2}$ and $h_{1}=h_{2}$. Therefore $x=r+h$ is unique.
\end{proof}

As a consequence of Theorem \ref{thm: finite}, since every element in $\mathcal{F}(^*\mathbb R)$ is of the form $r+h$, then the elements in  ${\mathcal{F}({^*\mathbb R})} / {\mathcal{I}({^*\mathbb R})}$ are nothing more than equivalence classes of real numbers. Therefore,  ${\mathcal{F}({^*\mathbb R})} / {\mathcal{I}({^*\mathbb R})}$ is isomorphic to $\mathbb R$.\\

The above theorem justifies the following definition.

\begin{definition}[Standard Part Mapping]
The mapping $st:\mathcal{F}(^*\mathbb R)\rightarrow\mathbb R$, defined by $st(x) = r$, where $x\in\mathcal{F}(^*\mathbb R)$ and $x=r+h$, is called the {\bf{standard part mapping}}. This mapping is also known as the canonical homomorphism between $\mathcal{F}(^*\mathbb R)$ and $\mathbb R$.
\end{definition}

\begin{theorem}
The standard part mapping is an order preserving homomorphism in the sense that, for finite x and y, we have $x\le y$ implies $st(x)\le st(y)$.
\end{theorem}

\begin{proof}
It is easy to show that the standard part mapping is a homomorphism. To verify that the mapping is order preserving, suppose (to the contrary) that $r_{1} > r_{2}$. Then $r_{1} + h_{1} \le r_{2} + h_{2}$ implies $0<r_{1} - r_{2} \le h_{2} - h_{1}$. Thus $r_{1} - r_{2}$ must be a real infinitesimal, or in other words, $r_{1} - r_{2} = 0$, a contradiction.
\end{proof}

\begin{remark}
The standard part mapping does not preserve strict inequalities. Indeed, $r<r+h$, for any real $r$ and any positive infinitesimal $h$. It is clear that $r<r+h$, but $st(r)=st(r+h)=r$.
\end{remark}
Ultimately, we shall use the standard part mapping as a means of characterizing limit convergence of sequences and functions in the nonstandard analysis.

\section{Extending Sets and Functions in $^*\mathbb R$}
Before we can begin our treatment of the nonstandard analysis, we must first define what it means for an object to be either a subset or a function of $^*\mathbb R$.

\begin{definition}\label{Def: extension set}
Let $A\subseteq\mathbb R$. Then the set $^*A=\{ {\bra a_{n} \ket} \in {^*\mathbb R} : a_{n} \in A \text{ a.e\}}$
is called the nonstandard extension of $A$.
\end{definition}

\begin{remark}
When we say that $a_{n}\in A$ almost everywhere (a.e), we mean that $\{n\in\mathbb N : a_{n} \in A\}\in\mathcal{U}$.
\end{remark}

The following theorems will establish the key properties of our new sets.

\begin{theorem}\label{Thm: set equality}
Let $A\subseteq\mathbb R$. Then $A\subseteq {^*A}$, with equality holding if and only if $A$ is finite.
\end{theorem}

\begin{proof}
\begin{description}

We begin by proving that $A\subseteq {^*A}$. Let $A\subseteq\mathbb R$ and let $a\in A$ where $a=\bra a, a, a, \dots \ket$. 
Then clearly, $\{n\in\mathbb N : a\in A\}=\mathbb N\in\mathcal{U}$, implying that $a\in{^*A}.$
 Therefore $A\subseteq {^*A}$.

 We conclude by showing that equality holds only if $A$ is finite.
\item ($\Rightarrow$) Suppose (to the contrary) that $A$ is not finite. Our goal is to construct an element in $^*A$ that is not in $A$. 
Consider the sequence $\bra a_{1}, a_{2}, a_{3}, \dots \ket$, where each term is a distinct element in $A$. Then
\[
\{n\in\mathbb N : a_{n}\in A\}\in\mathcal{U} \text{ (by construction)}.
\]
Thus $ {\bra a_{n} \ket} \in {^*A}$, but there does not exists an $a\in A$ such that $ {\bra a_{n}\ket} = a.$ Indeed, 
\[
\{ n\in\mathbb N : a = a_{n}\} = \varnothing\text{ or a singleton set,} 
\]
neither of which can be in $\mathcal{U}$. Therefore $A\not = {^*A}$, a contradiction.
($\Leftarrow$)
Let $A\subseteq\mathbb R$ such that $A=\{b_{1}, b_{2},\dots,b_{k}\}$ is finte, and let ${\bra a_{n} \ket}\in{^*A}$.
Thus $\{n\in\mathbb N : a_{n}\in A\}\in\mathcal{U}$, which can be rewritten as 
\[
\{n\in\mathbb N: a_{n}\in A\} = \{n\in\mathbb N: a_{n} = b_{1}\}  \cup  \{n\in\mathbb N : a_{n} = b_{2}\}  \cup  \dots  \cup 
\{n\in\mathbb N : a_{n} =b_{ k}\}.
\]
Since the left hand side is in $\mathcal{U}$, we know by Lemma \ref{Thm: union}
  that there must exist an $i\in\mathbb N$ such that $\{n\in\mathbb N : a_{n} = b_{i}\}\in\mathcal{U}$. Thus 
$ {\bra a_{n} \ket} = b_{i}\in A,$ and by part {\bf{(i)}} we conclude that $A= {^*A}$.
\end{description}
\end{proof}

\begin{theorem}\label{Thm: nonstandard}
Any infinite subset of $\mathbb R$ has nonstandard elements in its extension.
\end{theorem}

\begin{proof}
 Let $A\subseteq\mathbb R$ such that A is infinite. We construct the sequence, $ {\bra a_{1}, a_{2}, a_{3},\dots\ket},$ where 
$a_{i},a_{j}\in A$ for all $i,j\in\mathbb N$ and $a_{i} = a_{j}$ only if $i=j$. 
Then $\{n\in\mathbb N : a_{n}\in A\} = \mathbb N\in\mathcal{U}$, implying that ${\bra a_{n} \ket}\in{^*A}$. But for each $a\in A,$
\[
\{n\in\mathbb N : a_{n} = a\}
\]
is either empty or a singleton, both of which cannot be in $\mathcal{U}$ by Definition 5.1. 
Thus ${\bra a_{n} \ket} \in {^*A}\setminus A$ and therefore ${^*A}$ contains nonstandard elements.
\end{proof}

\begin{theorem}\label{Thm: boolean}
The boolean properties of sets are preserved by their nonstandard extensions.
\begin{description}
\item[(i)] $A\subseteq B$ if and only if ${^*A}\subseteq{^*B}$.
\item[(ii)] $^*{(A\cap B)} = {^*A}\cap{^*B}$.
\item[(iii)] $^*{(A\cup B)} = {^*A}\cup {^*B}$.
\item[(iv)] $^*{(A\setminus B)} = {^*A}\setminus {^*B}$.
\end{description}
\end{theorem}

\begin{proof}
\hspace{.5in}
\begin{description}
\item[(i)]$(\Rightarrow)$ Let $A$ and $B$ be subsets of $\mathbb R$ such that $A\subseteq B$. If ${\bra a_{n} \ket} \in {^*A}$, then 
$\{n\in\mathbb N : a_{n} \in A\}\in\mathcal{U}$. Conside the set $\{n\in\mathbb N : a_{n} \in B\}$. Since $A\subseteq B$, then
\[
\{n\in\mathbb N : a_{n}\in A\} \subseteq \{n\in\mathbb N : a_{n} \in B\}.
\]
By definition 5.1 (ii), we know that $\{n\in\mathbb N : a_{n} \in B \}\in\mathcal{U}$. Thus $ {\bra a_{n} \ket} \in {^*B}$, 
implying that $ {^*A} \subseteq {^*B}$. 

$(\Leftarrow)$ Let ${^*A}$ and ${^*B}$ be subsets of ${^*\mathbb R}$ such that ${^*A}\subseteq{^*B}$. If $a\in A$, then 
${\bra a,a,a,\dots\ket}\in{^*A}$. By assumption we have ${\bra a,a,a,\dots\ket}\in{^*B}$, which implies that 
\[
\{n\in\mathbb N : a\in B\}\in\mathcal{U}.
\] 
Thus $a\in B$ and we conclude that $A\subseteq B$.

\item[(ii)] $(\subseteq)$ Let ${ \bra c_{n} \ket } \in {^*(A\cap B)}$. Then we define the set $C=\{n\in\mathbb N : c_{n}\in {(A\cap B)}\}\in\mathcal{U}$. On the other hand, let 
\[
A=\{n\in\mathbb N : c_{n} \in A\} \text{ and } B=\{n\in\mathbb N : c_{n}\in B\}.
\]
Observe that $C=(A\cap B)$. Thus
\[
\{n\in\mathbb N : c_{n} \in A\} \cap \{n\in\mathbb N : c_{n} \in B\} \in\mathcal{U}.
\]
By definition 5.1(iii), we know that both $A$ and $B$ must be contained in $\mathcal{U}$. Thus 
${\bra c_{n} \ket} \in {^*A}\cap {^*B}$.

$(\supseteq)$ Let ${\bra c_{n} \ket}\in {^*A}\cap{^*B}$. Then $A\in\mathcal{U}$ and $B\in\mathcal{U}$.
 By definition 5.1 (ii), we know that $A\cap B\in\mathcal{U}$, but $C=A\cap B$.
 Thus $\{n\in\mathbb N : c_{n}\in A\cap B\}\in\mathcal{U}$, which implies that ${ \bra c_{n} \ket } \in {^*(A\cap B)}$.

\item The proofs for parts {\bf{(iii)}} and {\bf{(iv)}} are similar in the sense that they depend upon the properties of the free ultrafilter $\mathcal{U}$.
\end{description}
\end{proof}

The following result shall be usefull for our future work on characterizing the nonstandard definition of sequence convergence.

\begin{theorem}
${^*\mathbb N_{\infty}} = {^*\mathbb N}\setminus\mathbb N$ only contains infinitely large numbers.
\end{theorem}

\begin{proof}
 Our goal is to show that ${^*\mathbb N}$ does not contain infinitesimal numbers or nonstandard finite numbers, thus leaving us 
 with ${^*\mathbb N_{\infty}}= \mathcal{L}(^*\mathbb N)$.
Since $\mathbb N$ is an infinite subset of $\mathbb R$ we know, by Theorem \ref{Thm: nonstandard}, that 
${^*\mathbb N}$ contains nonstandard numbers, i.e. ${^*\mathbb N}\setminus{\mathbb N}\not = \varnothing$.
 It is easy to see that ${^*\mathbb N}$ cannot contain infinitesimal numbers and thus $\mathcal{I}(^*\mathbb N) = \varnothing$. 
In regards to $\mathcal{F}(^*\mathbb N)$, let $x\in\mathcal{F}(^*\mathbb N)$ such that $x=r+h$, where $r\in\mathbb N$ and $h\in\mathcal{I}(^*\mathbb R)$. Clearly $ r\le x < {r+1}$. 
If $x$ were strictly greater than $r$, there would exists a natural number between $r$ and $r+1$, a contradiction. Thus $x$ must equal $r$, implying that 
$h=0\in\mathcal{I}(^*\mathbb R)$. Thus $\mathcal{F}(^*\mathbb N) = \mathbb N$. Therefore, by process of elimination,  ${^*\mathbb N_{\infty}}= \mathcal{L}(^*\mathbb N)$.
\end{proof}

We conclude this section by defining what it means to extend a function  
$f:X\rightarrow\mathbb R$ to $^*f:{^*X}\rightarrow{^*\mathbb R}$.

\begin{definition}\label{Def: function}
Let $f:X\rightarrow\mathbb R$ be a real valued function where $X\subseteq\mathbb R$. Then the function, 
$^*f:{^*X}\rightarrow{^*\mathbb R}$, defined by ${^*f( {\bra x_{n} \ket})} = {\bra f(x_{n}) \ket}$ for all 
${\bra x_{n} \ket}\in{^*X}$ is called the nonstandard extension of $f$.
\end{definition}

The nonstandard extension of $f$ is a well defined function. In general
\[
\{n\in\mathbb N : x_{n} = {x^\prime}_{n}\}\subseteq \{n\in\mathbb N : f(x_{n}) = f({x^\prime}_{n})\}.
\]
Thus if $\{n\in\mathbb N : x_{n} = {x^\prime}_{n}\}\in\mathcal{U}$, then 
$\{n\in\mathbb N : f(x_{n}) = f({x^\prime}_{n})\}\in\mathcal{U}$. Therefore $^*f$ is a well defined function. 
Furthermore, ${^*f}$ agrees with $f$ on $\mathbb R$ in the sense that if $x\in\mathbb R$, then ${^*f(x) = f(x)}$.

\begin{example}
Recall from real analysis that the function, $f:\mathbb N\rightarrow\mathbb R$, defines a sequence in $\mathbb R$ such that 
$f(n) = a_{n}$. The nonstandard extension of the sequence is $^*f: {^*\mathbb N} \rightarrow {^*\mathbb R}$, where 
${^*f(n)} = {^*a_{n}}$. We do not call ${^*(a_{n})}$ a sequence since $\card({\mathbb N}) \not = \card({^*\mathbb N})$. Instead, we define ${^*(a_{n})}$ to be a {\bf{hypersequence}}. 
\end{example}

\section{Non-Standard Characterization of Limits in $\mathbb R$}
The next result belongs to A.Robinson~\cite{aRob66} and establishes the limit of a sequence  of real numbers in terms of non-standard numbers.

\begin{theorem}[Robinson]
The following are equivalent:
\begin{description}
\item[(i)] $\lim_{n\rightarrow\infty} a_{n} = L$ in the sense that
\[
(\forall\epsilon\in\mathbb R_{+})(\exists\nu\in\mathbb N)(\forall n\in\mathbb N)(n\ge\nu\implies |a_{n} - L|<\epsilon).
\]
\item[(ii)] $(\forall\omega\in{^*\mathbb{N}_{\infty}}) ({^*a_{n}}\approx L)$.
\end{description}
\end{theorem}

\begin{proof}
\hspace{.5in}
\begin{description}

\item[(i)$\Rightarrow$ (ii)] Assume that $\lim_{n\rightarrow\infty} a_{n} = L$. Let $\epsilon\in\mathbb R_{+}$ be fixed so that there exists a $\nu\in\mathbb N$ such that
\[
(\forall n\in\mathbb N)(n\ge\nu\implies |a_{n} - L|<\epsilon).
\]
Let $\omega\in {^*\mathbb N_{\infty}}$, where $\omega = \bra \omega_{1}, \omega_{2},\dots\ket$ $(\omega_{i}\in\mathbb N\text{ } \forall i\in\mathbb N)$.
Consider the hypersequence
\[
{^*f}:{^*\mathbb N}\rightarrow {^*\mathbb R},
\]
where ${^*f}$ is the non-standard extension of $(a_{n})$. By definition,\\ 
${^*f}(\omega) = {^*f}(\bra \omega_{1}, \omega_{2},\dots\ket) = \bra f(\omega_{1}), f(\omega_{2}),\dots\ket = \bra a_{\omega_{1}}, a_{\omega_{2}},\dots\ket = {^*a_{\omega}} $.
Then $|{^*a_{\omega}} - L| = \bra |a_{\omega_{1}} - L|, |a_{\omega_{2}}-L|,\dots\ket$. Since $\omega$ is infinitely large, we know that there exists an $i\in\mathbb N$ such that 
\[
\nu\le {\omega_{i}}< {\omega_{i+1}}< {\omega_{i+2}}<\dots
\]
Then by assumption,
\[
\{i\in\mathbb N: |a_{\omega_{i}} - L|<\epsilon\}\in\mathcal{U}.
\]
Thus $ | {^*a}_{\omega} - L| <\epsilon$, or in other words, ${^*a}_{\omega}\approx L$.

\item[(ii)$\Rightarrow$(i)] Assume that $({^*a}_{\omega} \approx L) (\forall \omega \in {^*\mathbb N_{\infty}})$. Suppose (to the contrary) that $\lim_{n\rightarrow\infty} a_{n} \not= L$. This implies that 
\[
(\exists\epsilon\in\mathbb R_{+})(\forall\nu\in\mathbb N)(\exists n\in\mathbb N)(n\ge\nu\text{ and } |a_{n} - L| >\epsilon).
\]
Thus there exists an infinite subset of $\mathbb N$ containing $\{n, n+1, n+2,\dots\}$, such that the above is true. We construct an infinitely large natural number of the form $\omega= \bra n, n+1, n+2,\dots\ket\in{^*\mathbb N_{\infty}}$. Obviously,
${^*a}_{\omega}\not\approx L$, contradicting our assumption.

\end{description}
\end{proof}


\appendix
\chapter{The Free Ultrafilter as an Additive Measure}

For those more familiar with measure theory, we can characterize the free ultrafilter as a finitely additive measure.

\begin{definition}\label{D: Measure}
Let $\mu$ denote a fixed, finitely additive measure on the set $\mathbb N$ such that:
\begin{enumerate}
\item $( \forall A\in\mathcal{P}(\mathbb N) ) (\mu(A)=1 \text{ or } \mu(A)=0)$.
\item $\mu(\mathbb N)=1$ and $\mu(A)=0$ for all finite $A\in\mathcal{P}(\mathbb N)$.
\end{enumerate}
When we say that $\mu$ is a finitely additive measure we mean that for all mutually disjoint $A,B\in\mathcal{P}(\mathbb N)$
\[
\mu(A+B)=\mu(A)+\mu(B).
\]
\end{definition}

\begin{lemma}[Properties of $\mu$] Let $\mu$ be as defined in definition \ref{D: Measure}. Then
\begin{description}
\item[(i)] Let $A\in\mathcal{P}(\mathbb N).$ Then either $\mu(A)=1$ or $\mu(\mathbb N-A)=1$,but not both.
\item[(ii)] Let $A,B\in\mathcal{P}(\mathbb N)$ such that if $\mu(A)=1$ and $\mu(B)=1$ then $\mu(A\cap B)=1$.
\item[(iii)] Let $B\in\mathcal{P}(\mathbb N)$ and let $A\in\mathcal{P}(\mathbb N)$ such that $A\subseteq B \text{ and } \mu(A)=1$. Then $\mu(B)=1$\newpage
\end{description}
\end{lemma}

\begin{proof}
\hspace{.5in}
\begin{description}
\item[(i)] Let $A\in\mathcal{P}(\mathbb N)$. Suppose (to the contrary) that $\mu(A)=1 \text{ and }\newline \mu(\mathbb N -A)=1$. 
Then $\mu(A)+\mu(\mathbb N-A) = \mu(A\cup(\mathbb N - A))=\mu(\mathbb N)=2$, contradicting the fact that $\mu(\mathbb N)=1$.

\item[(ii)] Let $A,B\in\mathcal{P}(\mathbb N)$ such that $\mu(A)=1 \text{ and } \mu(B)=1$. Taking the measure of the complement of $A\cap B$ we get
\[
\mu((A\cap B)^c)=\mu(A^c \cup B^c) \le \mu(A^c) + \mu(B^c),
\]
where $\mu(A^c) \text{ and } \mu(B^c)$ are both zero. Thus $\mu(A^c \cup B^c)=0$, which means that the complement, $\mu(A\cap B)=1$

\item[(iii)] Let $A,B\in\mathcal{P}(\mathbb N)$ such that $A\subseteq B \text{ and } \mu(A)=1$. Suppose (to the contrary) that $\mu(B)=0$. Then
\[
\mu(B)=\mu(A\cup (B-A))=\mu(A)+\mu(B-A)\ge 1,
\]
which is a contradiction, regardless of the measure of $(B-A)$.
\end{description}
\end{proof}
Observe that the properties of the measure $\mu$ are similar to those of the free ultrafilter $\mathcal{U}$ defined in Section 
\ref{S: Filters}. This is no coincidence, as the following theorem shall now demonstrate.

\begin{theorem}[Characterization of $\mu$]
Let $\mathcal{U}$ be a free ultrafilter on $\mathbb N$ and let $A\in\mathcal{P}(\mathbb N)$. Then
\begin{description}
\item[(i)] $A\in\mathcal{U}$ if and only if $\mu(A)=1$.
\item[(ii)]$A\not\in\mathcal{U}$ if and only if $\mu(A)=0$.
\end{description}
\end{theorem}

\begin{proof} 
\begin{description}
\item[(i)]
\begin{description}
\item[($\Rightarrow$)]
Let $A\in\mathcal{U}$. Suppose (to the contrary) that $\mu(A)=0$. Then
this implies that $\mu(\mathbb N) =0$, which is not possible since $\mu(\mathbb N)$ is defined to be 1.
\item[($\Leftarrow$)]
Let $A\in\mathcal{P}(\mathbb N) \text{ such that  } \mu(A)=1 \text{ and let } \mathcal{U} \text{ be a free ultrafilter on } \mathbb N$. Suppose (to the contrary) that $A\not\in\mathcal{U}$. Then
\[
\mu(A\cup(\mathbb N - A))=\mu(A)+\mu(\mathbb N - A)=1,
\]
which implies that $\mathbb N \not\in\mathcal{U}$, a contradiction.
\end{description}

\item[(ii)] The proof for $\bf{(ii)}$ is omitted since it is similar to the proof of $\bf{(i)}$.
\end{description}
\end{proof}



\end{document}